\documentclass[a4paper,10pt,reqno, english]{amsart} 
\usepackage{amssymb}
\usepackage{graphicx}
\usepackage{amsmath,amsthm}
\usepackage{amsfonts,amssymb,enumerate}
\usepackage{url,paralist}
\usepackage[arrow,curve,matrix,tips,2cell]{xy}  
  \SelectTips{eu}{10} \UseTips
  \UseAllTwocells

\usepackage{tikz}
\usepackage{pdfsync}
\usepackage{hyperref}
\usepackage{color}
\usepackage{enumerate,amssymb}

\theoremstyle{plain}
\newtheorem{theorem}{Theorem}[section]

\newtheorem{lemma}[theorem]{Lemma}

\newtheorem{corollary}[theorem]{Corollary}

\newtheorem*{theorem*}{Theorem}

\theoremstyle{definition}
\newtheorem{definition}[theorem]{Definition}

\newtheorem{example}[theorem]{Example}
\newtheorem{remark}[theorem]{Remark}
\newtheorem{notation}[theorem]{Notation}

\newcommand{\BIGOP}[1]{\mathop{\mathchoice%
{\raise-0.22em\hbox{\huge $#1$}}%
{\raise-0.05em\hbox{\Large $#1$}}{\hbox{\large $#1$}}{#1}}}

\newcommand{\RR}{\mathbb{R}}
\newcommand{\CC}{\mathbb{C}}

\newcommand{\ZZ}{\mathbb{Z}}
\newcommand{\FF}{\mathbb{F}}
\newcommand{\BB}{\mathrm{B}}
\newcommand{\EE}{\mathrm{E}}
\newcommand{\im}{\mathrm{im}}
\newcommand\RP{\mathbb{R}{\rm P}}

\newcommand\Sym{\mathfrak S}
\newcommand{\dimaff}{\dim_{\affine}}

\newcommand{\cat}{\operatorname{cat}}
\newcommand{\id}{\operatorname{id}}

\newcommand{\res}{\operatorname{res}}

\newcommand{\affine}{\operatorname{aff}}

\newcounter{commentcounter}

\begin{document}

\title{On highly regular embeddings}

\author{Pavle V. M. Blagojevi\'{c}} 
\thanks{The research by Pavle V. M. Blagojevi\'{c} leading to these results has
        received funding from the European Research Council under the European Union's Seventh
        Framework Programme (FP7/2007-2013) / ERC Grant agreement no.~247029-SDModels.  Also
        supported by the grant ON 174008 of the Serbian Ministry of Education and Science.}
\address{Mathemati\v cki Institut SANU, Knez Mihailova 36, 11001 Beograd, Serbia}
\email{pavleb@mi.sanu.ac.rs} 
\address{Institut f\" ur Mathematik, FU Berlin, Arnimallee 2, 14195 Berlin, Germany}
\email{blagojevic@math.fu-berlin.de}\urladdr{http://page.mi.fu-berlin.de/blagojevic/} 
\author{Wolfgang L\"uck} 
\thanks{The research by Wolfgang L\"uck leading to these results has received funding from the Leibniz
        Award granted by the DFG}
\address{Mathematisches Institut der Universit\"at Bonn, Endenicher Allee 60, 53115 Bonn, Germany} \email{wolfgang.lueck@him.uni-bonn.de}
\urladdr{http://www.him.uni-bonn.de/lueck/} 
\author{G\"unter M. Ziegler} 
\thanks{The research by G\"unter M. Ziegler leading to these results has received funding from the 
        European Research Council under the European Union's Seventh Framework Programme (FP7/2007-2013) / ERC   Grant agreement no.~247029-SDModels
        and by the DFG Collaborative Research Center TRR~109 ``Discretization in Geometry and Dynamics.''}  
\address{Institut f\" ur Mathematik, FU Berlin, Arnimallee 2, 14195 Berlin, Germany} \email{ziegler@math.fu-berlin.de}
\urladdr{http://page.mi.fu-berlin.de/gmziegler/} \date{\today}

\begin{abstract} \noindent
A continuous map $\RR^d\to\RR^N$ is \emph{$k$-regular} if it maps  any $k$ pairwise distinct points to $k$ linearly independent vectors.
Our main result on $k$-regular maps is the following lower bound for the existence of such maps between Euclidean spaces, in which 
$\alpha (k)$ denotes the number of ones in the dyadic expansion of~$k$:
\begin{quote}
    \emph{For $d\geq 1$ and $k\geq 1$ there is no $k$-regular map $\RR^d\to\RR^N$ for 
        $N<d(k-\alpha (k))+\alpha (k)$.}
\end{quote}
This reproduces a result of Cohen \& Handel from 1978 for $d=2$ and the extension by Chisholm from 1979 to the case when $d$ is a power of~$2$; 
for the other values of $d$ our bounds are in general better than Karasev's~(2010), who had only recently gone beyond
Chisholm's special case.
In particular, our lower bound turns out to be tight for $k\le3$.

A framework of Cohen \& Handel~(1979) relates the existence of a $k$-regular map to the existence of a low-dimensional inverse of a certain vector bundle. Thus
the non-existence of regular maps into $\RR^N$ for small~$N$ follows from the non-vanising of specific
dual Stiefel--Whitney classes. This we prove using the general Borsuk--Ulam--Bourgin--Yang theorem combined with a key observation by Hung~(1990) about the cohomology 
algebras of configuration spaces.

Our study produces similar lower bound results also for the 
existence of \emph{$\ell$-skew embeddings} $\RR^d\to\RR^N$, for which we require that the
images of the tangent spaces of any $\ell$ distinct points are skew affine subspaces.
This extends work by Ghomi \& Tabachnikov (2008) for the case $\ell=2$.
\end{abstract}


\maketitle

\section{Introduction and statement of the main results}

A $k$-regular embedding $X\to\RR^N$ maps any $k$ pairwise distinct points in a topological
space $X$ to $k$ linearly independent vectors.  The study of the existence of $k$-regular
maps was initiated by Borsuk~\cite{borsuk} in 1957 and later attracted additional
attention due to its connection with approximation theory.  The problem was extensively
studied by Chisholm~\cite{chisholm}, Cohen \& Handel~\cite{cohen-handel},
Handel~\cite{handel}, and Handel \& Segal~\cite{handel-segal} in the 1970's and 1980's.
In the 1990's the study of $k$-regular maps, and in particular the related notion of
$k$-neighbourly submanifolds, was continued by Handel~\cite{handel96} and
Vassiliev~\cite{vassiliev92}.  The most complete result from that time, which gives a
lower bound for the existence of $k$-regular maps between Euclidean spaces, is the
following result of Chisholm \cite[Theorem 2]{chisholm}:
\begin{compactitem}[\quad]
\item[] {\em For $d$ a power of $2$, $k\geq 1$, there is no $k$-regular map
    $\RR^d\to\RR^{d(k-\alpha (k))+\alpha (k)-1}$, where $\alpha (k)$ denotes the number of
    ones in the dyadic expansion of $k$.}
\end{compactitem}
This result was only recently extended by Karasev~\cite[Corollary~9.4
and~9.6]{karasev:regular embeddings} beyond the case when $d$ is a power of $2$.

The framework of Cohen \& Handel~\cite{cohen-handel} relates the existence of a
$k$-regular map to the existence of a specific inverse of the regular representation vector bundle over the unordered configuration space.
Using Stiefel--Whitney classes of this vector bundle, combined with a key observation by Hung~\cite{hung}, we
here get an extension of the Chisholm result with explicit lower bounds for \emph{all}
values of $d$; it will appear below as Theorem~\ref{theorem:Main-1}:
\begin{compactitem}[\quad]
\item[] {\em For any $d\geq 1$ and $k\geq 1$ there is no $k$-regular map
    $\RR^d\to\RR^{d(k-\alpha (k))+\alpha (k)-1}$.}
\end{compactitem}
This reproduces Chisholm's bound for the case of $d$ being a power of $2$, while for other
values of $d$ our bounds are in general better than Karasev's.  In particular, our lower
bound will turn out to be tight for $k=3$.  We mention without giving details that our
methods can also be used to get rid of the assumption that $k$ is a power of $2$ in the
theorem of Vassiliev appearing in~\cite[Theorem~1]{vassiliev}.

\smallskip A smooth embedding $M\to\RR^N$ of a manifold $M$ is an $\ell$-skew embedding if for
any $\ell$ pairwise distinct points on $M$ the corresponding tangent spaces of the image
in~$\RR^N$ are skew.  The notion of $\ell$-skew embeddings is a natural extension
of the notion of totally skew embeddings ($2$-skew embeddings) as introduced and studied in 2008
by Ghomi \& Tabachnikov~\cite{ghomi-tabachnikov}.  The existence of $2$-skew embeddings was
studied in a number of concrete examples by Barali\'c et
al.~\cite{baralic-prvulovic-stojanovic-vrecica-zivaljevic}.  Following the same pattern as
in the case of the $k$-regular maps, we get a new lower bound for the existence of
$\ell$-skew embeddings that will appear later as Theorem~\ref{theorem:Main-2}:
\begin{compactitem}[\quad]
\item[] \emph{For $d\geq 1$, $\ell\geq 1$ and $\gamma(d)=\lfloor\log_2d\rfloor+1$ there is no $\ell$-skew embedding 
\[\RR^d\to\RR^{2^{\gamma(d)}(\ell-\alpha(\ell))+(d+1)\alpha(\ell)-2}.\]}
\end{compactitem}
In particular, for $\ell=2$ our bound contains the bound
of Ghomi \& Tabachnikov~\cite[Theorem~1.4 and Corollary~1.5]{ghomi-tabachnikov}.

\smallskip

The concept of $k$-regular-$\ell$-skew embeddings combines the notions of $k$-regular maps and
$\ell$-skew embeddings.  It was introduced and studied in 2006 by Stojanovi\'c
\cite{stojanovic}.  Based on our results for $k$-regular and $\ell$-skew embeddings, we derive a
new lower bound for the existence of $k$-regular-$\ell$-skew embeddings between Euclidean spaces
that considerably improves other known lower bounds and appears as
Theorem~\ref{theorem:Main-3}:
\begin{compactitem}[\quad]
 \item[] {\em For any $k,\ell,d\geq 2$ there is no $k$-regular-$\ell$-skew embedding
          \[\RR^d\to\RR^{(d-1)(k-\alpha(k))+(2^{\gamma(d)}-d-1)(\ell-\alpha(\ell))+(d+1)\ell+k-2}.\]}
\end{compactitem}

\smallskip

The main technical points of this paper are related to the study of the dual Stiefel--Whitney classes of the regular representation vector bundle over the unordered configuration 
space: see the proof of Lemma~\ref{lemma:dualSW-1}, Corollary~\ref{cor:SW-monomials} and the proof of Theorem~\ref{theorem:Main-2.1}.
Many related calculations were performed in the classical literature. 
In particular, the map $H_*(\BB\Sym_k)\to H_*(\BB\mathrm{O}(k))$ was studied by Kochman~\cite{kochman} in the language of Dyer--Lashof operations.
Frederick Cohen, in his landmark paper from 1976 \cite{cohen}, described the map $H_*(F(\RR^d,k))\to H_*(\BB\Sym_k)$.
\medskip

\noindent
{\bf Acknowledgments.} Thanks to Fred Cohen for thoughtful remarks that initiated this study. 
We are grateful to Peter Landweber for valuable discussions and comments
and to the referee for helpful observations and references.

\section{$k$-regular maps}

In this section we will define and then study $k$-regular maps, review relevant known results and eventually prove the
following extension of the result by Chisholm \cite[Theorem~2]{chisholm}.
\begin{theorem}
  \label{theorem:Main-1}
  Let $k,d\geq 1$ be integers.  There is no $k$-regular map $\RR^d\to\RR^{N}$ for \[N\leq
  d(k-\alpha (k))+\alpha (k)-1,\] where $\alpha (k)$ denotes the number of ones in the
  dyadic expansion of $k$.
\end{theorem}

This result is easily seen to be true and tight for the cases $d=1$, $k=1$, and $k=2$.
As we will see it also gives the complete answer in the case of $3$-regular maps.
This same fact was observed by Handel in~\cite[Proposition 2.3]{handel79}.

\begin{corollary}
\label{cor:1}
A $3$-regular map $\RR^d\to\RR^N$ exists if and only if $N\geq d+2$.
\end{corollary}

In the case $d=2$ of $k$-regular maps from the plane we get the following complete answer in
the case when $k$ is power of $2$.

\begin{corollary}
\label{cor:2}
Let $m\geq 1$ be an integer. 
 Then a $2^m$-regular map $\RR^2\to\RR^N$ exists if and only if $N\geq 2^{m+1}-1$.
\end{corollary}

\subsection{Definition and first bounds}

All topological spaces we consider are Hausdorff spaces and all maps are continuous.

The \emph{configuration space} of $n$ ordered pairwise distinct points in the topological
space $X$ is the subspace of $X^k$ defined by
\[
F(X,k)=\{(x_1,\ldots,x_k)\in X^k : x_i\neq x_j\text{ for all }i\neq j\}.
\]
The symmetric group $\Sym_k$ acts freely on the configuration space by permuting the
points.

\begin{definition}[Regular maps]
  Let $X$ be a topological space, $k\geq 1$ be an integer, and $f\colon X\to\RR^N$ be a
  continuous map.  Then we say that the map $f$ is
  \begin{enumerate}[(1)]
  \item {\bf $k$-regular map} if for every $(x_1,\ldots,x_k)\in F(X,k)$ the vectors
    $f(x_1),\ldots,f(x_k)$ are linearly independent, and
  \item {\bf affinely $k$-regular map} if for every $(x_0,\ldots,x_k)\in F(X,k+1)$ the
    points $f(x_0),\ldots,f(x_k)$ are affinely independent.
  \end{enumerate}
\end{definition}

\noindent
Obviously, each $k$-regular map is also an affinely $(k-1)$-regular map.
Moreover the following lemma is a direct consequence of the definition.
\begin{lemma}
  \label{lemma:equivalence}
  A map $f\colon X\to\RR^N$ is affinely $(k-1)$-regular if and only if the map $g\colon
  X\to\RR\times\RR^N$, defined by $g(x)=(1,f(x))$, is $k$-regular.
\end{lemma}

\begin{example}
  \label{exa:special_maps}
    \ 
    \begin{enumerate}[(1)]
     
    \item \label{exa:special_maps:(1)} 
      The map $f_{\RR}\colon\RR\to\RR^k$ given by $f_{\RR}(x)=(1,x,x^2,\ldots,x^{k-1})$  and 
      the map $f_{\CC}\colon\CC\to\RR\times \CC^{k-1}$ given by 
      $f_{\CC}(z)=(1,z,z^2,\ldots,z^{k-1})$ are $k$-regular maps due to the nonvanishing
      of the Vandermonde determinant at every point of $F(\RR,k)$ and $F(\CC,k)$.
      
    \item \label{exa:special_maps:(2)} The standard embedding $i\colon S^n\to\RR^{n+1}$ is
      affinely $2$-regular.  Indeed, no affine line in $\RR^{n+1}$ intersects the sphere
      $S^n:=i(S^n)=\{x\in\RR^{n+1}:\|x\|=1\}$ in more than two points.  Thus, for every
      $(x_1,x_2,x_3)\in F(S^n,3)$ the set of points $\{i(x_1),i(x_2),i(x_3)\}$ cannot be
      on a single line, i.e., it is affinely independent.
      
    \item \label{exa:special_maps:(3)} If $X$ embeds into $S^n$, then there exists a $3$-regular map 
     $X\to\RR^{n+2}$,~\cite[Theorem 2.3]{handel-segal}.  Indeed, by the previous example there exists an
      affinely $2$-regular map $i\colon S^n\to\RR^{n+1}$.  Then by 
      Lemma~\ref{lemma:equivalence} the map $j \colon S^n\to\RR\times\RR^{n+1}$ given by
      $j(x)=(1,i(x))$ is a $3$-regular map.  In particular, there exists a $3$-regular map
      $\RR^n\to\RR^{n+2}$.
    \end{enumerate}
  
\end{example}

\begin{notation}\label{not:tuples} In the sequel we often abbreviate a tuple $(x_1,x_2, \ldots.,x_k)$
by $\underline{x}$, and analogously for $\underline{y}$ and $\underline{\lambda}$.
\end{notation}

The first necessary condition for the existence of $k$-regular maps between Euclidean
spaces was given by Boltjanski{\u\i}, Ry{\v{s}}kov \& {\v{S}}a{\v{s}}kin 
in~\cite{boltjanskii-ryskov-saskin}.

\begin{theorem}[Boltjanski{\u\i}, Ry{\v{s}}kov, {\v{S}}a{\v{s}}kin, 1963]
\label{the:BRS}
If there exists a $2k$-regular map $f\colon\RR^d\to\RR^N$, then $(d+1)k\leq N$.
\end{theorem}

\begin{proof}
Let $f\colon\RR^d\to\RR^N$ be a $2k$-regular map.
Consider $k$ pairwise disjoint non-empty open balls $D_1,\ldots,D_k$ in $\RR^d$ and 
let $g \colon D_1\times\cdots\times D_k\times (\RR{\setminus}\{0\})^k\to\RR^N$,
$(\underline{x}, \underline{\lambda}) \mapsto \sum_{i=1}^k\lambda_if(x_i)$.
This map $g$ is injective: Indeed, let us assume that
$g(\underline{x},\underline{\lambda}) =  g(\underline{y},\underline{\mu})$, or, equivalently,
that $\sum_{i=1}^k\lambda_if(x_i)-\sum_{i=1}^k\mu_if(y_i) = 0$.
This linear combination can be rewritten in the form
$
\sum_{z\in Z}\gamma_zf(z)=0,
$,
where $Z$ is the union of $\{x_1, x_2, \ldots , x_k\}$ and  $\{y_1, y_2, \ldots , y_k\}$,
$\gamma_z = \lambda_i$, if $z = x_i$ and $x_i \not = y_i$,
$\gamma_z = \mu_i$, if $z = y_i$ and $x_i \not = y_i$,
and $\gamma_z = \lambda_i - \mu_i$ if $z = x_i = y_i$. 
Since $\mathrm{card} (Z)\leq 2k$, the $2k$-regularity of $f$ implies that $\gamma_z=0$ for
all $z\in Z$.  So $(\underline{x},\underline{\lambda}) = (\underline{y},\underline{\mu})$
and $g$ is injective.
This implies that
$dk+k=\dim (D_1\times\cdots\times D_k\times (\RR{\setminus}\{0\})^k)\leq \dim\RR^N=N$.
\end{proof}

\subsection{A topological criterion}

In order to obtain better bounds we apply more elaborate tools.
First we introduce the Stiefel manifold of $k$-frames in a Euclidean space.

\begin{definition}[Stiefel manifold]
Let $1\leq k\leq N$ be integers.
The Stiefel manifold $V_k(\RR^N)$ of all ordered $k$-frames is a subset of the product $(\RR^N)^k$ given by
\[
V_k(\RR^N)=\{(y_1,\ldots,y_k)\in (\RR^N)^k : y_1,\ldots,y_k\text{ are linearly independent}\},
\]
and equipped with the subspace topology.
\end{definition}
\noindent
The symmetric group $\Sym_k$ acts freely on the Stiefel manifold by permuting the vectors in the frame, that is,
the columns of the matrix $(y_1,\ldots,y_k)\in (\RR^{N\times k})$.

\smallskip

With this we can formulate the following elementary but essential lemma.
It is a direct consequence of the definition of a $k$-regular map.
\begin{lemma}
\label{lem:Existence_of_a_map}
 If there exists a $k$-regular map $X\to\RR^N$, then there exists a $\Sym_k$-equivariant map \[F(X,k)\to V_k(\RR^N).\]
\end{lemma}

Now we are in the realm of equivariant topology and study the existence of an
$\Sym_k$-equivariant map $F(X,k)\to V_k(\RR^N)$.  
For that purpose we use the following problem about the existence of an inverse bundle, which was formulated and 
proved to be equivalent by Cohen \& Handel in 1978.

Consider the Euclidean space $\RR^k$ as an $\Sym_k$-representation with the action given by
coordinate permutation.  Then the subspace $W_k=\{(a_1,\ldots, a_k)\in\RR^k :\sum a_i=0\}$ is an
$\Sym_k$-subrepresentation of $\RR^k$.  Let us introduce the following vector bundles over
the unordered configuration space
\begin{equation*}
 \begin{array}{llllll}
 \xi_{X,k} \colon & \RR^k & \longrightarrow & F(X,k)\times_{\Sym_k}\RR^k &\longrightarrow  & F(X,k)/\Sym_k
  \\
 \zeta_{X,k} \colon & W_k   & \longrightarrow  & F(X,k)\times_{\Sym_k}W_k   &\longrightarrow  & F(X,k)/\Sym_k
  \\
 \tau_{X,k} \colon & \RR   & \longrightarrow  & F(X,k)/\Sym_k\times\RR     &\longrightarrow  & F(X,k)/\Sym_k
 \end{array}
\end{equation*}
where the last vector bundle is a trivial line bundle.
There is an obvious decomposition:
\[
\xi_{X,k}\cong\zeta_{X,k}\oplus \tau_{X,k}.
\]

\begin{lemma}[Cohen \& Handel, \cite{cohen-handel}]
  \label{lemma:equi}
  An $\Sym_k$-equivariant map $F(X,k)\to V_k(\RR^N)$ exists if and only if the
  $k$-dimensional vector bundle $\xi_{X,k}$ admits an $(N-k)$-dimensional inverse.
\end{lemma}

\begin{proof}
($\Longleftarrow$):
Let $\eta$ be an $(N-k)$-dimensional inverse of $\xi_{X,k}$.
Then the composition $f$ of the inclusion (of total spaces) of vector bundles followed by the projection to the fiber (of a trivial vector bundle):
\[
f\colon\xi_{X,k}\rightarrow \xi_{X,k}\oplus\eta\cong\tau^N_{X,k} = F(X,k)/\Sym_k\times\RR^N\rightarrow\RR^N
\]
when restricted on each fiber is a linear monomorphism.  \
Using the map $f$ we define the required $\Sym_k$-equivariant map
\[
g\colon F(X,k) \to V_k(\RR^N), \quad \underline{x} \mapsto  \bigl(f(\underline{x}, e_1),\ldots, f(\underline{x},e_k)\bigr),
\]
where $\{e_1,\ldots,e_k\}$ denotes the standard basis of $\RR^k$.

\noindent
($\Longrightarrow$):
Let $g\colon F(X,k)\to V_k(\RR^N)$ be an $\Sym_k$-equivariant map.
Consider the following map
\[
h'\colon V_k(\RR^N)\times_{\Sym_k}\RR^k\to\RR^N, \quad 
(\underline{y},\underline{\lambda}) \mapsto \sum_{i=1}^{k}\lambda_i y_i.
\]
It is a linear monomorphism on each fiber of the vector bundle
\[
 \nu \colon \RR^k\to V_k(\RR^N)\times_{\Sym_k}\RR^k\to V_k(\RR^N)/\Sym_k.
\]
Thus $h'$ induces a fiberwise injective map $h \colon \nu \to \theta$, where $\theta$ is
the trivial bundle over $V_k(\RR^N)/\Sym_k$ with the fiber $\RR^N$.
The maps $g$ and $h$ induce the following composition of morphisms of vector bundles:
\[
\xymatrix@!C=9em{F(X,k)\times_{\Sym_k}\RR^k  \ar[r]^{g\times_{\Sym_k}\id}  \ar[d]^{\xi_{X,k}}
& 
V_k(\RR^N)\times_{\Sym_k}\RR^k    \ar[r]^h  \ar[d] ^{\nu}    
& 
V_k(\RR^N)/\Sym_k\times\RR^N \ar[d]^{\theta}
\\
F(X,k)/\Sym_k    \ar[r]_{g/_{\Sym_k}}                     
& 
V_k(\RR^N)/\Sym_k   \ar[r]_{\id} 
&
V_k(\RR^N)/\Sym_k
}
\]
Since the morphism induced by $h$ is a linear monomorphism on each fiber, we get that
$\mathrm{coker}\big(h\colon\nu\to\theta\big)$ is a vector bundle,~\cite[Corollary~8.3, page~36]{husemoller}.  
Thus, the pullback bundle
\[
 (\mathrm{id}\circ g/_{\Sym_k})^*\mathrm{coker}\big(h\colon\nu\to\theta\big)
\]
is the required $(N-k)$-dimensional inverse of the vector bundle $\xi_{X,k}$.
\end{proof}

Now we formulate an immediate consequence of Lemmas~\ref{lem:Existence_of_a_map} 
and~\ref{lemma:equi} that gives us a direct criterion for the non-existence of a $k$-regular
map.

\begin{lemma}\
  \label{lemma:criterion}

  \begin{enumerate}[\rm (1)]
  \item \label{lemma:criterion:no_map} If $\xi_{X,k}$ does not admit an $r$-dimensional inverse, then there cannot be any
    $k$-regular map $X\to\RR^{k+r}$.

  \item \label{lemma:criterion:STW}
    If the dual Stiefel--Whitney class $\overline{w}_{r+1}(\xi_{X,k})$ does not vanish,
    then there cannot be any $k$-regular map $X\to\RR^{k+r}$.
  \end{enumerate}
\end{lemma}

Now according to Lemma~\ref{lemma:criterion}~\eqref{lemma:criterion:STW} we see that Theorem~\ref{theorem:Main-1}
is the consequence of the following result.

\begin{theorem}
\label{theorem:Main-1.1}
Let $k,d\geq 1$ be integers.
Then the  dual Stiefel--Whitney class $\overline{w}_{(d-1)(k-\alpha(k))}(\xi_{\RR^d,k})$ does not vanish.
\end{theorem}

Using the connectivity of the Stiefel manifold $V_k(\RR^N)$ 
and the criterion of Cohen \& Handel (Lemma~\ref{lemma:equi}) this 
theorem yields an interesting consequence.

\begin{corollary}
If $k$ is a power of~$2$, then an $\Sym_k$-equivariant map 
\[F(\RR^d,k)\rightarrow V_k(\RR^N)\]
exists if and only if $N\ge (d-1)(k-1)+k$.    
\end{corollary}

\subsection{Proof of Theorem~\ref{theorem:Main-1.1}}
\label{subsec:Proof_of_Main_1.1}

Theorem~\ref{theorem:Main-1.1} will be proved in two steps, first when $k$ is a power of $2$
(Lemma~\ref{lemma:dualSW-1}) and then for all $k\geq 1$ (Lemma~\ref{lemma:dualSW-2}).
We start with the following extension of~\cite[Lemma~3]{chisholm} and~\cite[Lemma~3.2]{cohen-handel}.
All our cohomology groups are understood to be with $\FF_2$ coefficients.

\begin{lemma}
\label{lemma:dualSW-1}
Let $d\geq 1$ be an integer, and $k=2^m$ for some $m\geq 1$.  Then the dual
Stiefel--Whitney class $\overline{w}_{(d-1)(k-1)}(\xi_{\RR^d,k})$ does not vanish.
\end{lemma}
\begin{proof}
\textbf{Step 1.}
Let $E_m =(\ZZ/2)^m$ be the subgroup of $\Sym_k$ given by the regular embedding $(\mathrm{reg})\colon E_m\to\Sym_k$,~\cite[Example~2.7, page~100]{adem-milgram}.
The regular embedding is determined by the left translation action of $(\ZZ/2)^m$ on itself. 
 To each $g\in (\ZZ/2)^m$ we associate a permutation $L_g\colon  (\ZZ/2)^m\to  (\ZZ/2)^m$ from $\mathrm{Sym}((\ZZ/2)^m)\cong\Sym_k$ given by $L_g(x)=g+x$.
 
\noindent
The image of the restriction of $H^*(\Sym_k)$ from $\Sym_k$ to $E_m$ is $H^*(E_m)^{\mathrm{GL}_m(\FF_2)}$.
There are specific elements  $q_{m,s}$  in $H^{2^m-2^s}(E_m)^{\mathrm{GL}_m(\FF_2)}$ 
for $0\leq s\leq m-1$, called the {\em Dickson invariants}, such that $H^*(E_m)^{\mathrm{GL}_m(\FF_2)}$
is isomorphic to $\FF_2[q_{m,m-1},\ldots,q_{m,0}]$ as a graded $\FF_2$-algebra,
see~\cite[Chapter~3.E on page~57ff]{madsen-milgram}. Let $\langle q_{m,0}\rangle$ denote the
ideal generated by the top Dickson invariant $q_{m,0}$ inside the polynomial ring
$\FF_2[q_{m,m-1},\ldots,q_{m,0}]$. Then we obtain a sequence of isomorphisms of graded $\FF_2$-modules
\begin{eqnarray*}
H^*(\Sym_k)    &\cong& \ker(\res^{\Sym_k}_{E_m})\oplus \im (\res^{\Sym_k}_{E_m})\\
                     &\cong& \ker(\res^{\Sym_k}_{E_m})\oplus H^*(E_m)^{\mathrm{GL}_m(\FF_2)}\\
                     &\cong& \ker(\res^{\Sym_k}_{E_m})\oplus \FF_2[q_{m,m-1},\ldots,q_{m,0}]\\
                     &\cong& \ker(\res^{\Sym_k}_{E_m})\oplus \FF_2[q_{m,m-1},\ldots,q_{m,1}]\oplus\langle q_{m,0}\rangle.
\end{eqnarray*}

\medskip
\noindent
\textbf{Step 2.} Let $\eta_k$ be the vector bundle given by the Borel construction for the
permutation action of the symmetric group $\Sym_k$ on $\RR^k$:
\[
\RR^k\longrightarrow \EE\Sym_k\times_{\Sym_k}\RR^k\longrightarrow \BB\Sym_k.
\]
When $k=2^m$ is a power of $2$, we conclude from~\cite[Lemma~3.26 in Chapter~3.E on page~59]{madsen-milgram} that
\[
  w_i(\eta_k )
    \begin{cases}
     =q_{m,s},                          & \text{for} \; i=2^m-2^s\, \text{and} \; 0\leq s\leq m-1;\\
      \in \ker(\res^{\Sym_k}_{E_m}),          & \text{otherwise}.
    \end{cases}
    \]

\smallskip
\noindent
Now consider the vector bundle $\xi_{\RR^d,k}$ and the corresponding classifying map 
\[\alpha_{d,k}\colon F(\RR^d,k)/\Sym_k\to\BB\Sym_k\] 
associated to the free $\Sym_k$-space $F(\RR^d,k)$.
It induces a pullback morphism of vector bundles
$\xi_{\RR^d,k}\to\eta_k$.  Thus the Stiefel--Whitney classes of the vector bundle
$\xi_{\RR^d,k}$ are given by:
\begin{multline}
\label{eq:sw-class-1}
w_i:=w_i(\xi_{\RR^d,k})=\alpha_{d,k}^*w_i(\eta_k)
    \\
    \begin{cases}
     =0,                                                  & \text{for} \; i\ge 2^m=k,\\
     =\alpha_{d,k}^*q_{m,s},                              & \text{for} \; i=2^m-2^s\; \text{and}\; 0\leq s\leq m-1,\\
     \in \alpha_{d,k}^*(\ker(\res^{\Sym_k}_{E_m})),        & \text{otherwise}
    \end{cases}
\end{multline}
 
Here are three additional known facts on Stiefel--Whitney classes of the vector bundles $\xi_{\RR^d,k}$ and $\eta_k$ that we will use:
\begin{itemize}[$\bullet$]
\item From~\cite[Lemma~8.14]{blagojevic-luck-ziegler} we have that
  \begin{equation}
    \label{eq:sw-class-2}
    0\neq w_{(d-1)(k-1)}(\xi_{\RR^d,k}^{\oplus(d-1)})=w_{k-1}^{d-1}\in
    H^{(d-1)(k-1)}(F(\RR^d,k)/\Sym_k),
  \end{equation}
  or in another words, since $w_{k-1}^{d-1}=\alpha_{d,k}^*(q_{m,0}^{d-1})$,
  \begin{equation}
    \label{eq:sw-class-2-FH}
    q_{m,0}^{d-1}\notin
    \mathrm{Index}_{\Sym_k}(F(\RR^d,k)):=\ker \alpha_{d,k}^*;
  \end{equation}

\item Moreover, as in~\cite[Corollary~4.4]{blagojevic-ziegler-convex}, specializing to
  $\FF_2$ coefficients we have
  \begin{equation}
    \label{eq:sw-class-2.5}
    H^{(d-1)(k-1)}(F(\RR^d,k)/\Sym_k)=\langle w_{k-1}^{d-1}\rangle
    =
    \langle (\alpha_{d,k}^*q_{m,0})^{d-1}\rangle\cong\FF_2,
  \end{equation}
  and in dimensions $i>(d-1)(k-1)$ we get $H^i(F(\RR^d,k)/\Sym_k)=0$;

\item The following decomposition of $H^*(F(\RR^d,k)/\Sym_k)$ was proved by Hung in~\cite[(4.7), page~279]{hung}:
  \begin{multline}
    \label{eq:sw-class-3}
    H^*(F(\RR^d,k)/\Sym_k)
    \\
    \cong
    \alpha_{d,k}^*\big(\ker(\res^{\Sym_k}_{E_m})\oplus \FF_2[q_{m,m-1},\ldots,q_{m,1}]\big)
    \oplus 
    \alpha_{d,k}^*(\langle q_{m,0}\rangle).
  \end{multline}
  In particular, this implies that
  \[
   \alpha_{d,k}^*\big(\ker(\res^{\Sym_k}_{E_m})\oplus \FF_2[q_{m,m-1},\ldots,q_{m,1}]\big)
   \cap
   \alpha_{d,k}^*(\langle q_{m,0}\rangle)
   =\{0\}.
  \]

\end{itemize}
Directly from~\eqref{eq:sw-class-2.5} and~\eqref{eq:sw-class-3} we conclude that:
\begin{multline*}
\label{eq:sw-class-4}
u\in \ker(\res^{\Sym_k}_{E_m})\oplus \FF_2[q_{m,m-1},\ldots,q_{m,1}]
\; \text{and} \; 
\deg u\geq (d-1)(k-1)
\\
\Longrightarrow
u\in \mathrm{Index}_{\Sym_k}(F(\RR^d,k)).
\end{multline*}
Thus, for all $j_1,\ldots,j_{k-2}\geq 0$, such that $\sum_{r=1}^{k-2}r\cdot j_r\geq (d-1)(k-1)$:
\begin{equation}
\label{eq:sw-class-4.1}
w_1^{j_1}\cdots w_{k-2}^{j_{k-2}}=0\in H^*(F(\RR^d,k)/\Sym_k),
\end{equation}
or, equivalently,  in the notation of the Fadell--Husseini index \cite{fadell-husseini}: 
\begin{equation}
\label{eq:sw-class-4.11}
w_1(\eta_k)^{j_1} \cdots  w_{k-2}(\eta_k)^{j_{k-2}}\in \mathrm{Index}_{\Sym_k}(F(\RR^d,k)).
\end{equation}

\medskip
\noindent
\textbf{Step 3.}
Next we prove that for all $j_1,\ldots,j_{k-2}\geq 0$ and $1\leq j_{k-1}\leq d-2$
such that $\sum_{r=1}^{k-1}r\cdot j_r\geq (d-1)(k-1)$ we have
\begin{equation}
\label{eq:sw-class-4.6}
w_1^{j_1}\cdots w_{k-2}^{j_{k-2}}w_{k-1}^{j_{k-1}}=0\in H^*(F(\RR^d,k)/\Sym_k).
\end{equation}
or equivalently
\begin{equation}
w_1(\eta_k)^{j_1}\cdots w_{k-2}(\eta_k)^{j_{k-2}}w_{k-1}(\eta_k)^{j_{k-1}}\in\mathrm{Index}_{\Sym_k}(F(\RR^d,k)).
\label{eq:sw-class-4.789}
\end{equation}
In order to prove the equivalent equations~\eqref{eq:sw-class-4.6} and~\eqref{eq:sw-class-4.789}
we need the following claim which we will show next.
\begin{itemize}
\item[{\bf Claim.}]{\em Let $n\geq 2$ and $k=2^m$. Then
      \[
        w_{k-1}(\eta_k)\cdot \mathrm{Index}_{\Sym_k}(F(\RR^n,k))\subseteq \mathrm{Index}_{\Sym_k}(F(\RR^{n+1},k)).
        \]} 
 \end{itemize}
    The claim is a consequence of the general
      Borsuk--Ulam--Bourgin--Yang theorem~\cite[Section~6.1]{blagojevic-luck-ziegler},
      applied to
      \begin{compactitem}[$\bullet$]
      \item the $\Sym_k$-equivariant map $\phi\colon F(\RR^{n+1},k)\to W_k$ that is the
        composition of the obvious inclusion $F(\RR^{n+1},k)\to (\RR^{n+1})^k$ and the two orthogonal
        projections $(\RR^{n+1})^k\to (\RR\times\{0\}\times\cdots\times\{0\})^k=\RR^k$ and
        $\RR^k\to W_k$, and
      \item the set $Z:=\{0\}\subseteq W_k$.
      \end{compactitem}
      Thus by the general Borsuk--Ulam--Bourgin--Yang theorem we get:
      \[
       \mathrm{Index}_{\Sym_k}(W_k{\setminus}Z)\cdot\mathrm{Index}_{\Sym_k}(\phi^{-1}(Z))
      \subseteq \mathrm{Index}_{\Sym_k}(F(\RR^{n+1},k)).
      \]
      First, $W_k{\setminus}Z=W_k{\setminus}\{0\}$ is $\Sym_k$-homotopy equivalent to the sphere $S(W_k)$.
      Thus \[\mathrm{Index}_{\Sym_k}(W_k{\setminus}\{0\})
     =\mathrm{Index}_{\Sym_k}(S(W_k))=\langle e(\nu_k)\rangle,\]
      where $e(\nu_k)$ is the Euler class, with $\FF_2$-coefficients, of the vector bundle $\nu_k$:
      \[
       W_k\longrightarrow \EE\Sym_k\times_{\Sym_k}\RR^k\longrightarrow \BB\Sym_k,
      \]
      see~\cite[Proof of Proposition~3.11, page~1338]{blagojevic-ziegler}. In this
      case, due to $\FF_2$
      coefficients, \[e(\nu_k)=w_{k-1}(\nu_k)=w_{k-1}(\eta_k)=q_{m,0}.\] The inverse image
      $\phi^{-1}(Z)$ is the set of points $(x_1,\ldots,x_k)\in F(\RR^{n+1},k)$ in the
      configuration space such that all first coordinates of the points
      $x_1,\ldots,x_k$ are equal.  Thus we can identify $\phi^{-1}(Z)$ with
      $\RR\times F(\RR^n,k)\simeq_{\Sym_k} F(\RR^n,k)$.\newline 
      All these facts imply that:
      \[
       \langle w_{k-1}(\eta_k)\rangle\cdot \mathrm{Index}_{\Sym_k}(F(\RR^n,k))\ 
       \subseteq \mathrm{Index}_{\Sym_k}(F(\RR^{n+1},k)).
      \]
This finishes the proof of the claim above.

Let us assume that $j_1,\ldots,j_{k-2}\geq 0$ and $1\leq j_{k-1}\leq d-2$ with $\sum_{r=1}^{k-1}r\cdot j_r\geq (d-1)(k-1)$.
Thus $\sum_{r=1}^{k-2}r\cdot j_r\geq (d-1-j_{k-1})(k-1)$.
We conclude from~\eqref{eq:sw-class-4.11}
\begin{equation*}
w_1(\eta_k)^{j_1}\cdots w_{k-2}(\eta_k)^{j_{k-2}}\in\mathrm{Index}_{\Sym_k}(F(\RR^{d-j_{k-1}},k)).
\end{equation*}
Now the validity of the equivalent equations~\eqref{eq:sw-class-4.6} and~\eqref{eq:sw-class-4.789} follows from the claim.

\medskip
\noindent
\textbf{Step 4.} Finally, the dual Stiefel--Whitney class
$\overline{w}_{(d-1)(k-1)}:=\overline{w}_{(d-1)(k-1)}(\xi_{\RR^d,k})$, is defined by
$\overline{w}=\tfrac{1}{1+w_1+w_2+\cdots}=\sum_{n\geq 0}(w_1+w_2+\cdots)^n$.
The multinomial theorem implies the following presentation:
\begin{eqnarray*}
    \lefteqn{\overline{w}_{(d-1)(k-1)}}
    & &
    \\   
     &  =  &
      \sum_{\substack{j_{1},\ldots ,j_{k-1}\geq 0  \\ %
          j_{1}+2j_2+\cdots +(k-1)j_{k-1}=(d-1)(k-1)}}\binom{j_{1}+\cdots +j_{k-1}}{j_{1},\ j_{2},\
        \ldots, \ j_{k-1}}\, w_{1}^{j_{1}} \cdots  w_{k-1}^{j_{k-1}}\\
      \\
      &  =  &
      w_{k-1}^{d-1} +
      \sum_{\substack{j_{1},\ldots ,j_{k-2}\geq 0;\,d-2\geq j_{k-1}\geq 0  \\ %
          j_{1}+2j_2+\cdots +(k-1)j_{k-1}=(d-1)(k-1)}}\binom{j_{1}+\cdots +j_{k-1}}{j_{1},\ j_{2},\
        \ldots, \ j_{k-1}}\, w_{1}^{j_{1}} \cdots  w_{k-1}^{j_{k-1}},
\end{eqnarray*}
    where $\binom{j_{1}+\cdots +j_{k-1}}{j_{1},\ i_{2},\ \ldots, \ j_{k-1}}$ stands for the multinomial coefficient
  $\tfrac{\left( j_{1}+\cdots +j_{k-1}\right) {!}}{\left(j_{1}\right) {!} \cdots \left(
      j_{k-1}\right) {!}}$ modulo $2$.

\medskip

For $d$ and $k$ powers of $2$, Chisholm proved \cite[Lemma 3]{chisholm} that $\overline{w}_{(d-1)(k-1)}\neq 0$
by evaluating a pullback of  $\overline{w}_{(d-1)(k-1)}$ on a specific homology class $Q^I[1]$.
Here $Q^I$ denotes a Dyer--Lashof operation where 
\[
I=(2^{m-1}(d-1),\ldots,2(d-1),d-1)
\] 
is an admissible sequence of degree $(d-1)(k-1)$ and excess $d-1$.
First, he proved that the pullback of $w_{k-1}^{d-1}$ evaluated at $Q^I[1]$ is nonzero.
Then he verified that the pullback of $w_{1}^{j_{1}} \cdots  w_{k-1}^{j_{k-1}}$ evaluated at $Q^I[1]$ is zero if at least one of $j_i$ is nonzero for $i<k-1$ odd. 
Finally, he showed that all multinomial coefficients
\[
 \binom{j_{1}+\cdots +j_{k-1}}{j_{1},\ j_{2},\ \ldots, \ j_{k-1}} =0
\]
vanish in the case when $d$ is power of two, 
$0\leq j_{k-1}\leq d-2$ and $j_{1}=j_{3}=\cdots =j_{k-3}=0$ with
\[j_{1}+2j_2+\cdots +(k-1)j_{k-1}=(d-1)(k-1).\]
For details consult \cite[pages~188-189]{chisholm}.

\medskip

Here $d\geq 2$ is arbitrary  (but still $k=2^m$).
We use the fact~\eqref{eq:sw-class-2}, but instead of analyzing 
multinomial coefficients we consider the monomials in Stiefel--Whitney classes
\[
 w_{1}^{j_{1}} \cdots  w_{k-1}^{j_{k-1}},
\]
for $j_{1},\ldots ,j_{k-2}\geq 0$, $1\leq j_{k-1}\leq d-2$ and $j_{1}+2j_2+\cdots +(k-1)j_{k-1}=(d-1)(k-1)$.
According to~\eqref{eq:sw-class-4.1} and~\eqref{eq:sw-class-4.6} all these monomials vanish.
Therefore,
\[
 \overline{w}_{(d-1)(k-1)}= w_{k-1}^{d-1},
\]
which does not vanish, by \eqref{eq:sw-class-2}.
This finishes the proof of Lemma~\ref{lemma:dualSW-1}.
\end{proof}

For later purposes we  prove
\begin{corollary}
\label{cor:SW-monomials}
Consider a matrix $[j_{r,s}]_{1\leq r\leq t,1\leq s\leq k-1}$ of non-negative integers with pairwise distinct rows.
Assume that for some $0\leq j\leq d-1$ and each $1\leq r\leq t$
\[
 \sum_{s=1}^{k-1}s\cdot j_{r,s} = (k-1)j.
\]
Then, assuming the notation of Lemma~\ref{lemma:dualSW-1}, with $w_i:=w_i(\xi_{\RR^d,k})$, $d\geq 2$ and $k=2^m$:
\begin{equation}
\label{eq:eq}
 \sum_{r=1}^{t}\lambda_r\cdot w_1^{j_{r,1}}\cdots w_{k-2}^{j_{r,k-2}}w_{k-1}^{j_{r,k-1}}=w_{k-1}^{j}
\end{equation}
for some $\lambda_1,\ldots,\lambda_t\in\FF_2$ if and only if there exists a (unique) $r_0 \in \{1,2, \ldots , t\}$ 
such that
\begin{itemize}
 \item $\lambda_r=0$ if and only if $r\neq r_0$, and
 \item $j_{r_0,1}=\cdots=j_{r_0,k-2}=0,\,j_{r_0,k-1}=j$.
\end{itemize}
\end{corollary}
\begin{proof}
Multiplying the equation \eqref{eq:eq} with $w_{k-1}^{d-1-j}$ we get
\[
\sum_{r=1}^{t}\lambda_r\cdot w_1^{j_{r,1}}\cdots w_{k-2}^{j_{r,k-2}}w_{k-1}^{j_{r,k-1}+d-1-j}=w_{k-1}^{d-1}.
\]
The equation \eqref{eq:sw-class-2} implies that the right hand side of the equation does not vanish.
We conclude from  equations~\eqref{eq:sw-class-4.1} and~\eqref{eq:sw-class-4.6} 
that $w_1^{j_{r,1}}\cdots w_{k-2}^{j_{r,k-2}}w_{k-1}^{j_{r,k-1}+d-1-j}$ is different from zero only if
$j_{r,1}=\cdots=j_{r,k-2}=0,\,j_{r,k-1}=j$ holds. 
\end{proof}
In order to complete the proof of Theorem~\ref{theorem:Main-1} we extend Lemma~\ref{lemma:dualSW-1}  to all
$k\geq 1$ and prove the final fact. 

\begin{lemma}
\label{lemma:dualSW-2}
Let $d,k\geq 1$ be integers.
Then the dual Stiefel--Whitney class $\overline{w}_{(d-1)(k-\alpha(k))}(\xi_{\RR^d,k})$ does not vanish.
\end{lemma}
\begin{proof}
  Let $a:=\alpha(k)$ and $k=2^{r_1}+\cdots+2^{r_a}$ where $0\leq r_1<r_2<\cdots<r_a$.  We define
  a morphism of vector bundles $\prod_{t=1}^{a}\xi_{\RR^d,2^{r_t}}$ and $\xi_{\RR^d,k}$
  such that the following commutative square is a pullback diagram: 
 \[
\xymatrix{\prod_{t=1}^{a}\xi_{\RR^d,2^{r_t}} \ar[r]^-{\Theta}  \ar[d] 
& \xi_{\RR^d,k}   \ar[d] 
\\
\prod_{t=1}^{a}F(\RR^d,2^{r_t})/\Sym_{2^{r_t}}  \ar[r]_-{\theta}   & 
F(\RR^d,k)/{\Sym_k}.
}
\]
Choose embeddings $e_i \colon \RR^d \to \RR^d$ for $i = 1,2 \ldots , a$ such that their
images are pairwise disjoint open $d$-balls. They induces embeddings $F(\RR^d,\ell) \to F(\RR^d,\ell)$ denoted
by the same letter $e_i$ for all natural numbers $\ell$. The map $\theta$ is induced by the map
\[
\prod_{t=1}^{a}F(\RR^d,2^{r_t}) \to F(\RR^d,k),
\quad (\underline{x_1} ,\ldots , \underline{x_a})  \mapsto e_1(\underline{x_1}) \times \cdots \times e_a(\underline{x_a}).
\]
The map $\Theta$ is given by
\[
\bigl((\underline{x_1},v_1), \ldots ,(\underline{x_a},v_a)\bigr)
\mapsto
\bigl(e_1(\underline{x_1}) \times \cdots \times e_a(\underline{x_a}), v_1\times \cdots\times  v_a\bigr).
\]
Thus, the pullback bundle is a direct product bundle
\begin{eqnarray}
\theta^*\xi_{\RR^d,k} & \cong & \prod_{t=1}^{a}\xi_{\RR^d,2^{r_t}}.
\label{iso_Theta_xi_is_prod}
\end{eqnarray}

Now the naturality property of the Stiefel--Whitney classes~\cite[Axiom~2,
page~37]{milnor-stasheff} implies that in cohomology we get
\[
\theta^*\overline{w}_{(d-1)(k-a)}(\xi_{\RR^d,k})=\overline{w}_{(d-1)(k-a)}\Big(\prod_{t=1}^{a}\xi_{\RR^d,2^{r_t}}\Big).
\]
The product formula~\cite[Problem~4-A, page~54]{milnor-stasheff} gives us the following
equality of total dual Stiefel--Whitney classes
\[
 \overline{w}\Big(\prod_{t=1}^{a}\xi_{\RR^d,2^{r_t}}\Big)=\overline{w}(\xi_{\RR^d,2^{r_1}})\times\cdots\times\overline{w}(\xi_{\RR^d,2^{r_a}}).
\]
Consequently
\begin{eqnarray*}
 \theta^*\overline{w}_{(d-1)(k-a)}(\xi_{\RR^d,k}) 
& = & 
\overline{w}_{(d-1)(k-a)}\left(\prod_{t=1}^{a}\xi_{\RR^d,2^{r_t}}\right)
\\
& = &
\sum_{s_1+\cdots +s_a=(d-1)(k-a)}\overline{w}_{s_1}(\xi_{\RR^d,2^{r_1}})\times\cdots\times\overline{w}_{s_a}(\xi_{\RR^d,2^{r_a}}).
\end{eqnarray*}
Since each term $\overline{w}_{s_1}(\xi_{\RR^d,2^{r_1}})\times\cdots\times\overline{w}_{s_a}(\xi_{\RR^d,2^{r_a}})$ of the previous sum
belongs to a different direct summand of the cohomology, when the K\"unneth formula is applied,
\begin{multline*}
H^{(d-1)(k-a)}\Big(\prod_{t=1}^{a}F(\RR^d,2^{r_t})/\Sym_{2^{r_t}}\Big)\\
\cong
\bigoplus_{s_1+\cdots +s_a=(d-1)(k-a)}H^{s_1}(F(\RR^d,2^{r_1})/\Sym_{2^{r_1}})\otimes\cdots\otimes H^{s_a}(F(\RR^d,2^{r_a})/\Sym_{2^{r_a}})
\end{multline*}
we get the following criterion:
\begin{multline*}
\overline{w}_{(d-1)(k-a)}(\prod_{t=1}^{a}\xi_{\RR^d,2^{r_t}}) \neq 0 \Longleftrightarrow\\
\overline{w}_{s_1}(\xi_{\RR^d,2^{r_1}})\times\cdots\times\overline{w}_{s_a}(\xi_{\RR^d,2^{r_a}})\neq 0
\; \text{for some} \; s_1+\cdots +s_a=(d-1)(k-a).
\end{multline*}
By Lemma~\ref{lemma:dualSW-1} we have that
$\overline{w}_{(d-1)(2^{r_t}-1)}(\xi_{\RR^d,2^{r_t}})\neq 0$, and therefore
\begin{multline*}
0\neq  \overline{w}_{(d-1)(2^{r_1}-1)}(\xi_{\RR^d,2^{r_1}})\times\cdots\times\overline{w}_{(d-1)(2^{r_a}-1)}(\xi_{\RR^d,2^{r_a}})\in \\
 H^{(d-1)(k-a)}\Big(\prod_{t=1}^{a}F(\RR^d,2^{r_t})/\Sym_{2^{r_t}}\Big).
\end{multline*}
Thus, $\theta^*\overline{w}_{(d-1)(k-a)}(\xi_{\RR^d,k})=\overline{w}_{(d-1)(k-a)}(\prod_{t=1}^{a}\xi_{\RR^d,2^{r_t}})\neq 0$ and consequently
\[\overline{w}_{(d-1)(k-a)}(\xi_{\RR^d,k})\neq 0.\vspace{-6mm}\]
\end{proof}

\begin{remark}
The result $\overline{w}_{(d-1)(k-\alpha(k))}(\xi_{\RR^d,k})\neq 0$ of Lemma~\ref{lemma:dualSW-2} 
provides an alternative proof of Roth's result \cite[Theorem 1.4, page 449]{roth}
on the Lusternik--Schnirelmann category of the unordered configuration space:
\[
 \cat\left(F(\RR^d,k)/\Sym_k\right)\geq (d-1)(k-\alpha(k)).
\]
for every $d,k\geq 2$.
\end{remark}

Now, from Lemma~\ref{lemma:criterion}~\eqref{lemma:criterion:STW} and Lemma~\ref{lemma:dualSW-2},
 we conclude that there cannot be any $k$-regular map
\[
 \RR^d\to\RR^{d(k-\alpha (k))+\alpha (k)-1}.
\]
This finishes the proof of Theorem~\ref{theorem:Main-1}.

Finally, when Theorem~\ref{theorem:Main-1} is combined with 
Example~\ref{exa:special_maps}~\eqref{exa:special_maps:(3)} it implies
Corollary~\ref{cor:1}, and when
Example~\ref{exa:special_maps}~\eqref{exa:special_maps:(1)} is used,
we get Corollary~\ref{cor:2}.

\section{$\ell$-skew embeddings}

In this section we will first define $\ell$-skew embeddings, which were previously considered
for $\ell=2$ by Ghomi \& Tabachnikov~\cite{ghomi-tabachnikov}, and in even greater
generality by Stojanovi\'c \cite{stojanovic}.  We then prove the following Chisholm-like
theorem for $\ell$-skew embeddings.

\begin{theorem}
  \label{theorem:Main-2}
  Let $\ell,d\geq 2$ be integers.  There is no $\ell$-skew embedding $\RR^d\to\RR^{N} $
  for \[N\leq 2^{\gamma(d)}(\ell-\alpha(\ell))+(d+1)\alpha(\ell)-2,\] where
  $\alpha(\ell)$ denotes the number of ones in the dyadic expansion of $l$ and
  $\gamma(d)=\lfloor\log_2d\rfloor+1$.
\end{theorem}

In the notation of the paper by Stojanovi\'c~\cite{stojanovic} the claim of
Theorem~\ref{theorem:Main-2} can be stated as the following lower bound
\[
N_{0,\ell}(M)\geq N_{0,\ell}(\RR^d)\geq 2^{\gamma(d)}(\ell-\alpha(\ell))+(d+1)\alpha(\ell)-1,
\]
where $M$ is any $d$-manifold.  This bound, as illustrated in the table below, is a
considerable improvement of the bound $N_{0,\ell}(\RR^d)\geq (d+1)\ell-1$ obtained
in~\cite[Remark~2.3]{stojanovic}.
\medskip
\begin{center}
\begin{tabular}{| c | l | c | c | c | c | c | c | c |}
    \hline
 $\ell$ & Bounds                     & $d=2$ & $d=3$ & $d=4$ & $d=5$ & $d=6$ & $d=7$ & $d=8$  \\ \hline
     3  & $2^{\gamma(d)}+2d+1$       &  $9$  &  $11$ & $17$  & $19$  & $21$  & $23$  & $33$   \\
        & $3(d+1)-1$                 &  $8$  &  $11$ & $14$  & $17$  & $20$  & $23$  & $26$   \\ \hline
     4  & $3\cdot2^{\gamma(d)}+d$    &  $14$ &  $15$ & $28$  & $29$  & $30$  & $31$  & $56$   \\
        & $4(d+1)-1$                 &  $11$ &  $15$ & $19$  & $23$  & $27$  & $31$  & $35$   \\ \hline
     5  & $3\cdot2^{\gamma(d)}+2d+1$ &  $17$ &  $19$ & $33$  & $35$  & $37$  & $39$  & $65$   \\
        & $5(d+1)-1$                 &  $14$ &  $19$ & $24$  & $29$  & $24$  & $39$  & $44$   \\ \hline
        \end{tabular}
\end{center}
\medskip
Moreover, it improves even the bound $N_{0,\ell}(M)\geq
(d+1)\ell$,~\cite[Theorem~3.2]{stojanovic}, given for any closed manifold $M$, in all the
cases except when $d=2^r-1$ for some $r\geq 2$.

For $\ell=2$ the lower bound, in the notation of Ghomi \& Tabachnikov~\cite{ghomi-tabachnikov}, becomes
\[
N(d)\geq d+2^{\gamma(d)}.
\]
This bound was not explicitly given in~\cite[Theorem~1.4 and Corollary~1.5]{ghomi-tabachnikov} but could have been derived.
See Case~\ref{subsec:case_2} in our proof of Theorem~\ref{theorem:Main-2.1}.

\subsection{Definition and the first bound}

The affine subspaces $L_1,\ldots,L_{\ell}$ of the Euclidean space $\RR^N$ are
\emph{affinely independent} if the affine span of their union is an affine space of affine dimension 
$(\dimaff L_1+1)+\cdots+(\dimaff L_{\ell}+1)-1$.  Notice that any two lines
in $\RR^3$ are skew if and only if they are affinely independent.

For a $d$-dimensional manifold $M$ we denote by $TM$ the tangent bundle of $M$ and by
$T_yM$ the tangent space of $M$ at the point $y\in M$.

\begin{definition}[Skew embedding]\label{def:skew_map}
  Let $\ell\geq 1$ be an integer, and $M$ be a smooth $d$-dimensional manifold.  A smooth embedding
  $f \colon M\to\RR^N$ is an \emph{$\ell$-skew embedding} if for every $(y_1,\ldots,y_{\ell})\in F(M,\ell)$
  the affine subspaces
  \[
  (\iota\circ df_{y_1})(T_{y_1}M),\ldots ,(\iota\circ df_{y_{\ell}})(T_{y_{\ell}}M)
  \]
  of $\RR^N$ are affinely independent.
\end{definition}

\noindent
Here $df\colon TM\to T\RR^N$ denotes the differential map between tangent vector bundles
induced by $f$, and 
\begin{equation}
\iota \colon T\RR^N \to \RR^N
\label{iota}
\end{equation}
sends a tangent vector $v \in T_x\RR^N$ for $x \in \RR^N$
to  $x +v$ where we use the standard identification $T_x\RR^N = \RR^N$.


Now, directly from the definition we get the following lower bound for the existence of an
$\ell$-skew embedding.

\begin{lemma}
Let $\ell\geq 1$ be an integer, $M$ be a smooth $d$-dimensional manifold, and $f \colon M\to\RR^N$ be an $\ell$-skew embedding.
Then $(d+1)\ell-1\leq N$.
\end{lemma}
\begin{proof}
Since $f$ is an $\ell$-skew embedding, then for any $(y_1,\ldots,y_{\ell})\in F(M,\ell)$ the following inequality has to hold
\[
 (d+1)\ell-1=\dimaff\mathrm{span}\big\{(\iota\circ df_{y_1})(T_{y_1}M)\cup\cdots\cup (\iota\circ df_{y_{\ell}})(T_{y_{\ell}}M)\big\}\leq N.\vspace{-6mm}
\]
\end{proof}

\subsection{A topological criterion}

Now similarly to the case of $k$-regular maps, we derive a topological criterion for
the existence of an $\ell$-skew embedding.  Let $M$ be a smooth $d$-dimensional manifold, $TM$ be its
tangent bundle, and $\ell\geq 1$ be an integer.

The tangent manifold $TF(M,\ell)$ over the configuration space $F(M,\ell)\subseteq
M^{\ell}$ is the restriction of the direct sum of the pullback bundles
\[
TF(M,\ell)\cong\Big(\pi_1^* (TM)\oplus\cdots\oplus\pi_{\ell}^* (TM)\Big)\Big|_{F(M,\ell)},
\]
where $\pi_i\colon M^{\ell}\to M$ denotes the projection on the $i$th coordinate.

The symmetric group $\Sym_{\ell}$ acts naturally on the configuration space $F(M,\ell)$
and consequently on the tangent bundle $TF(M,\ell)$.  Since the action is free the
quotient space $TF(M,\ell)/\Sym_{\ell}$ can be identified with the tangent bundle of the
unordered configuration space $F(M,\ell)/\Sym_{\ell}$, i.e.,
$T(F(M,\ell)/\Sym_{\ell})\cong TF(M,\ell)/\Sym_{\ell}$.  For example, it is obvious that
\begin{equation}
\label{eq:TMofConfSpace}
T(F(\RR^d,\ell)/\Sym_{\ell})\cong d\,\xi_{\RR^d,\ell}.
\end{equation}

The first ingredient of our topological criterion is the existence of the following
fiberwise linear monomorphism.

\begin{lemma}
  \label{lemma:TK-l-1}
  Let $\ell\geq 1$ be an integer, and $M$ be a smooth $d$-dimensional manifold.
  If there exists an $\ell$-skew embedding $M\to\RR^N$, then the $(d+1)\ell$-dimensional vector
  bundle $T(F(M,\ell)/\Sym_{\ell})\oplus\xi_{M,\ell}$ over the unordered configuration space
  $F(M,\ell)/\Sym_{\ell}$ admits an $(N-(d+1)\ell+1)$-dimensional inverse.
\end{lemma}
\begin{proof}
Let us introduce the following two embeddings $a\colon\RR^N\to\RR^{N+1}$ and $b\colon\RR^N\to\RR^{N+1}$ defined by
\[
 a(t_1,\ldots,t_N)=(t_1,\ldots,t_N,1)\quad \text{and} \quad b(t_1,\ldots,t_N)=(t_1,\ldots,t_N,0).
\]
Now assume that $f\colon M\to\RR^N$ is an $\ell$-skew embedding.
Next we define a morphism of vector bundles covering the identity on the base space $F(M,\ell)/\Sym_{\ell}$
\[
\omega \colon T(F(M,\ell)/\Sym_{\ell})\oplus\xi_{M,\ell}\longrightarrow  F(M,\ell)/\Sym_{\ell}\times\RR^{N+1},
\]
as follows.
Let $p \colon F(M,\ell) \to F(M,\ell)/\Sym_{\ell}$ be the projection.
For $\underline{y} \in F(M,\ell)$ 
the linear map \[\omega_{p(\underline{y})} \colon T_{p(\underline{y})} (F(M,\ell)/\Sym_{\ell})\oplus (\xi_{M,\ell})_{p(\underline{y})}  \to \RR^{N+1}\]
is given by the formula 
\[
\bigl((\underline{y},\underline{v}), (\underline{y},\underline{\lambda})\bigr)
\mapsto \lambda_1 \cdot  (a\circ f(y_1)) + b\circ df_{y_1}(v_1)+
\cdots+\lambda_{\ell} \cdot (a\circ f(y_{\ell})) +b\circ df_{y_{\ell}}(v_{\ell}).
\]
where $\underline{v} \in TF(M,\ell)_{\underline{y}} \cong T_{y_1}M \oplus\cdots\oplus T_{y_{\ell}} M$
and $\underline{\lambda} \in \RR^{\ell}$.
Since $f$ is an $\ell$-skew embedding, this $\omega$ is a linear monomorphism on each fiber.
Hence the vector bundle $T(F(M,\ell)/\Sym_{\ell})\oplus\xi_{M,\ell}$ admits an $(N-(d+1)\ell+1)$-dimensional inverse.
\end{proof}

As a direct consequence of the Lemma~\ref{lemma:TK-l-1}, we get the following criterion for the
non-existence of an $\ell$-skew embedding $M\to\RR^N$.

\begin{lemma}
  \label{lemma:TK-l-2}
  Let $d,\ell\geq 1$ be integers and let $M$ be a smooth $d$-dimensional manifold.  
  If the dual Stiefel--Whitney  class 
  \[
   \overline{w}_{N - (d+1)\ell +2}\bigl(T(F(M,\ell)/\Sym_{\ell})\oplus\xi_{M,\ell}\bigr)
   \] 
   does not  vanish, then there is no $\ell$-skew embedding $M\to\RR^{N}$.
\end{lemma}

In the case when $M$ is the Euclidean space $\RR^d$ the relation~\eqref{eq:TMofConfSpace}
implies the following criterion.

\begin{lemma} \label{lem:dual:Whitney_skew}
  Let $d,\ell\geq 1$ be integers.  Suppose that the dual Stiefel--Whitney class
  $\overline{w}_{N - (d+1)\ell +2}((d+1)\,\xi_{\RR^d,\ell})$ does not vanish.
  Then there is no  $\ell$-skew embedding $\RR^d\to\RR^N$.
\end{lemma}

Thus Theorem~\ref{theorem:Main-2} is a consequence of Lemma~\ref{lem:dual:Whitney_skew} and the
following Theorem~\ref{theorem:Main-2.1}

\begin{theorem}
\label{theorem:Main-2.1}
Let $d,\ell\geq 1$ be integers.
Then the  dual Stiefel--Whitney class $\overline{w}_{(2^{\gamma(d)}-d-1)(\ell-\alpha(\ell))}\bigl((d+1)\,\xi_{\RR^d,\ell}\bigr)$
does not vanish.
\end{theorem}

\subsection{Proof of Theorem~\ref{theorem:Main-2.1}}

The proof of Theorem~\ref{theorem:Main-2.1} will be done in four steps, 
for increasing generality of parameters $d\geq 2$ and $\ell\geq 2$.
(For $d=1$ or $\ell=1$ the result is trivially true.)

\subsubsection{The special case $d=2$ and $\ell\geq 2$}
\label{subsec:case_1}

In this case we use the fact that $2\xi_{\RR^2,\ell}$ is a trivial vector
bundle~\cite[Theorem~1]{cohen-mahowald-milgram}.  Thus
\[
 \overline{w}_{(2^{\gamma(d)}-d-1)(\ell-\alpha(\ell))}((d+1)\,\xi_{\RR^d,\ell})
=\overline{w}_{\ell-\alpha(\ell)}(3\,\xi_{\RR^2,\ell})=\overline{w}_{\ell-\alpha(\ell)}(\xi_{\RR^2,\ell})\neq 0,
\]
by Lemma~\ref{lemma:dualSW-2}.

\subsubsection{The special case $d\geq 2$ and $\ell=2$}
\label{subsec:case_2}

In this case the base space of the vector bundle $\xi_{\RR^d,\ell}=\xi_{\RR^d,2}$ is homotopy equivalent to a
projective space, $F(\RR^d,2)/\Sym_2\simeq\RP^{d-1}$.  
More precisely, consider the inclusion 
$\RP^{d-1}\to F(\RR^d,2)/\Sym_2$ induced by another inclusion $S^{d-1}\to F(\RR^d,2)$ defined by $x\mapsto (x,-x)$.
Then the vector bundle $\xi_{\RR^d,2}$ over $F(\RR^d,2)/\Sym_2$ pulls back to the vector bundle isomorphic to
the direct sum of the tautological bundle and trivial line bundle over the projective space
$\RP^{d-1}$.  Thus, if $H^*(\RP^{d-1},\FF_2)\cong H^*( F(\RR^d,2)/\Sym_2,\FF_2)=\FF_2[w_1]/\langle w_1^d\rangle$ where
$\deg(w_1)=1$, then $w(\xi_{\RR^d,2})=1+w_1$. 
Consequently,
\[
 w((d+1)\,\xi_{\RR^d,2})=w(\xi_{\RR^d,2})^{d+1}=(1+w_1)^{d+1}.
\]
Since $2^{\gamma(d)}$ is a power of two and $2^{\gamma(d)}\geq d$ we have that
\begin{eqnarray*}
 w((d+1)\,\xi_{\RR^d,2})(1+w_1)^{2^{\gamma(d)}-d-1}
& = &
(1+w_1)^{d+1}(1+w_1)^{2^{\gamma(d)}-d-1}
\\
& = &
(1+w_1)^{2^{\gamma(d)}}=1+w_1^{2^{\gamma(d)}}
\\
& =& 
1,
\end{eqnarray*}
and therefore
\[
\overline{w}((d+1)\,\xi_{\RR^d,2})=(1+w_1)^{2^{\gamma(d)}-d-1}=1+\cdots+w_1^{2^{\gamma(d)}-d-1}.
\]
The inequality $2^{\gamma(d)}\leq 2d$ implies that $2^{\gamma(d)}-d-1<d$ and so
\[
 \overline{w}_{(2^{\gamma(d)}-d-1)(l-\alpha(l))}((d+1)\,\xi_{\RR^d,l})
=\overline{w}_{2^{\gamma(d)}-d-1}((d+1)\,\xi_{\RR^d,2})=w_1^{2^{\gamma(d)}-d-1}\neq 0.
\]
(This sort of argument is quite classical, see e.g.~\cite[Proposition 3.4, page 260]{howard}.)
\subsubsection{The special case $d\geq 2$ and $\ell=2^m$ for $m\geq 1$}
\label{subsubsec:1}
From the equation \eqref{eq:sw-class-1} we have that
\[
 w(\xi_{\RR^d,\ell})=1+w_1+\cdots+w_{\ell-1}=1+u,
\]
where $w_i:=w_i(\xi_{\RR^d,\ell})$, and $u:=w_1+\cdots+w_{\ell-1}$.
Then
\[
  w((d+1)\,\xi_{\RR^d,\ell})=w(\xi_{\RR^d,\ell})^{d+1}
=(1+w_1+\cdots+w_{\ell-1})^{d+1}=(1+u)^{d+1}.
\]
Let $\gamma(d,\ell):=\lfloor\log_2(d-1)+\log_2(\ell-1)\rfloor+1$ be 
the smallest power of $2$ that exceeds $(d-1)(\ell-1)$. Thus,
\begin{eqnarray*}
w((d+1)\,\xi_{\RR^d,\ell})(1+u)^{2^{\gamma(d,\ell)}-d-1}
& = &
(1+u)^{d+1}(1+u)^{2^{\gamma(d,\ell)}-d-1}
\\
& = &
(1+u)^{2^{\gamma(d,\ell)}}
\\
& = & 
1+u^{2^{\gamma(d,\ell)}}\\
& = & 
1+w_1^{2^{\gamma(d,\ell)}}+\cdots+w_{\ell-1}^{2^{\gamma(d,\ell)}}
\\
& = & 
1,
\end{eqnarray*}
since $2^{\gamma(d,\ell)}>(d-1)(\ell-1)$, and $H^i(F(\RR^d,\ell)/\Sym_{\ell})=0$ for all $i>(d-1)(\ell-1)$.
Therefore,
\begin{eqnarray*}
\overline{w}((d+1)\,\xi_{\RR^d,\ell}) &=&(1+u)^{2^{\gamma(d,\ell)}-d-1}\\
                                      &=&(1+w_1+\cdots+w_{\ell-1})^{2^{\gamma(d,\ell)}-d-1}\\
                                      &=&
 \sum_{\substack{k_0,k_1\ldots ,k_{\ell-1}\geq 0  \\ %
                 k_0+k_1+\cdots +k_{\ell-1}=2^{\gamma(d,\ell)}-d-1}}
 \binom{2^{\gamma(d,\ell)}-d-1}{k_0,\ k_1,\ \ldots, \ k_{\ell-1}}\, 
 w_1^{k_1} \cdots  w_{\ell-1}^{k_{\ell-1}}.
\end{eqnarray*}
The $(\ell-1)(2^{\gamma(d)}-d-1)$ component of this total dual Stiefel--Whitney class can be expressed in the following way:
\begin{multline*}
\overline{w}_{(\ell-1)(2^{\gamma(d)}-d-1)}((d+1)\,\xi_{\RR^d,\ell})\\
=
 \sum_{\substack{k_0,k_1\ldots ,k_{\ell-1}\geq 0  \\ %
                 k_0+k_1+\cdots +k_{\ell-1}=2^{\gamma(d,\ell)}-d-1\\
                 k_1+2k_2+\cdots +(\ell-1)k_{\ell-1}=(\ell-1)(2^{\gamma(d)}-d-1)}}
 \binom{2^{\gamma(d,\ell)}-d-1}{k_0,\ k_1,\ \ldots, \ k_{\ell-1}}\, 
 w_1^{k_1} \cdots  w_{\ell-1}^{k_{\ell-1}}.
\end{multline*}
For the choice of indices $k_0=2^{\gamma(d,\ell)}-2^{\gamma(d)},k_1=\cdots=k_{\ell-2}=0$, and $k_{\ell-1}=2^{\gamma(d)}-d-1$ we get a presentation
\begin{multline}
\label{eq:presentation-1}
\overline{w}_{(\ell-1)(2^{\gamma(d)}-d-1)}((d+1)\,\xi_{\RR^d,\ell})\\
=
\binom{2^{\gamma(d,\ell)}-d-1}{2^{\gamma(d,\ell)}-2^{\gamma(d)},\ 0,\ \ldots, \ 0,\ 2^{\gamma(d)}-d-1}
w_{\ell-1}^{2^{\gamma(d)}-d-1} + \mathrm{Rest},
\end{multline}
where $\mathrm{Rest}$ is a linear combination of monomials in Stiefel--Whitney classes different from the monomial $w_{\ell-1}^{2^{\gamma(d)}-d-1}$, i.e.,
\[
 \mathrm{Rest}:=
 \sum_{\substack{k_0,k_1\ldots ,k_{\ell-1}\geq 0,\, {k_{\ell-1}<2^{\gamma(d)}-d-1}  \\ %
                 k_0+k_1+\cdots +k_{\ell-1}=2^{\gamma(d,\ell)}-d-1\\
                 k_1+2k_2+\cdots +(\ell-1)k_{\ell-1}=(\ell-1)(2^{\gamma(d)}-d-1)}}
 \binom{2^{\gamma(d,\ell)}-d-1}{k_0,\ k_1,\ \ldots, \ k_{\ell-1}}\, 
 w_1^{k_1} \cdots  w_{\ell-1}^{k_{\ell-1}}.
\]

Since by Lucas' theorem \cite{lucas} from 1878 
\[
\binom{2^{\gamma(d,\ell)}-d-1}{2^{\gamma(d,\ell)}-2^{\gamma(d)},\ 0,\ \ldots, \ 0,\ 2^{\gamma(d)}-d-1}
=
\binom{2^{\gamma(d,\ell)}-d-1}{2^{\gamma(d)}-d-1}=1,
\]
and since in $\mathrm{Rest}$ only the monomials in Stiefel--Whitney classes that are different from the monomial
$w_{\ell-1}^{2^{\gamma(d)}-d-1}$ can appear,
 applying Corollary~\ref{cor:SW-monomials} we get that
\[
\overline{w}_{(\ell-1)(2^{\gamma(d)}-d-1)}((d+1)\,\xi_{\RR^d,\ell})\neq 0.
\]
This concludes the proof of this case. 

\subsubsection{The general case, $d\geq 2$ and $\ell\geq 2$}
In the final step we use the proof of Lemma~\ref{lemma:dualSW-2} combined with the result
of the previous special case~\ref{subsubsec:1}.

Let $a:=\alpha(\ell)$ and $\ell=2^{r_1}+\cdots+2^{r_a}$ where $0\leq r_1<r_2<\cdots<r_a$.
The isomorphism of vector bundles \eqref{iso_Theta_xi_is_prod} 
\[
\theta^*\xi_{\RR^d,\ell}\cong \prod_{t=1}^{a}\xi_{\RR^d,2^{r_t}}
\] 
implies that
\[
 \theta^*((d+1)\,\xi_{\RR^d,\ell})\cong \prod_{t=1}^{a}(d+1)\,\xi_{\RR^d,2^{r_t}}.
\]
Consequently,
\[
\theta^*\overline{w}\big((d+1)\,\xi_{\RR^d,\ell}\big)= \overline{w}\big(\prod_{t=1}^{a}(d+1)\,\xi_{\RR^d,2^{r_t}}\big)
=
\overline{w}((d+1)\,\xi_{\RR^d,2^{r_1}})\times\cdots\times\overline{w}((d+1)\,\xi_{\RR^d,2^{r_a}}).
\]
In order to simplify the formulas that follow we denote by:
\begin{compactitem}
\item $\sigma:=(2^{\gamma(d)}-d-1)(\ell-\alpha(\ell))$, and
\item $\sigma_i:=(2^{\gamma(d)}-d-1)(2^{r_i}-1)$ for all $1\leq i\leq a$.
\end{compactitem}
Now similarly to the proof of Lemma~\ref{lemma:dualSW-2}, using the presentation \eqref{eq:presentation-1}, we get
\begin{multline*}
 \theta^*\overline{w}_{\sigma}((d+1)\,\xi_{\RR^d,\ell}) \\
= \sum_{s_1+\cdots+s_a=\sigma} \overline{w}_{s_1}((d+1)\,\xi_{\RR^d,2^{r_1}})\times\cdots\times\overline{w}_{s_a}((d+1)\,\xi_{\RR^d,2^{r_a}}).
\end{multline*}
As in the proof of Lemma~\ref{lemma:dualSW-2}, each term in the previous sum belongs to
a different direct summand of the cohomology
\begin{multline*}
H^{\sigma}\Big(\prod_{t=1}^{a}F(\RR^d,2^{r_t})/\Sym_{2^{r_t}}\Big)\\
\cong
\bigoplus_{s_1+\cdots +s_a=\sigma}H^{s_1}(F(\RR^d,2^{r_1})/\Sym_{2^{r_1}})\otimes\cdots\otimes H^{s_a}(F(\RR^d,2^{r_a})/\Sym_{2^{r_a}}).
\end{multline*}
Thus
\begin{multline*}
\theta^*\overline{w}_{\sigma}((d+1)\,\xi_{\RR^d,\ell})\neq 0 \ \Leftrightarrow\\
 \overline{w}_{s_1}((d+1)\,\xi_{\RR^d,2^{r_1}})\times\cdots\times\overline{w}_{s_a}((d+1)\,\xi_{\RR^d,2^{r_a}})\neq 0\; 
\text{for some} \; s_1+\cdots+s_a=\sigma.
\end{multline*}
From the previous special case~\ref{subsubsec:1} we have that
\[
 \overline{w}_{\sigma_1}((d+1)\,\xi_{\RR^d,2^{r_1}})\times\cdots\times\overline{w}_{\sigma_a}((d+1)\,\xi_{\RR^d,2^{r_a}})\neq 0,
\]
and consequently $\theta^*\overline{w}_{\sigma}((d+1)\,\xi_{\RR^d,\ell})\neq 0$, implying that
\[
\overline{w}_{\sigma}((d+1)\,\xi_{\RR^d,\ell})\neq 0.
\]
This completes the proof of Theorem~\ref{theorem:Main-2.1}.

\section{$k$-regular-$\ell$-skew embeddings}
The notion of $k$-regular-$\ell$-skew embedding combines the notions of affinely
$k$-regular map and $\ell$-skew embedding.  It was introduced by Stojanovi\'c
in~\cite{stojanovic} and originally called $(k,\ell)$-regular maps.
In this section we will define $k$-regular-$\ell$-skew embeddings, derive some properties, and finally
prove the following common generalization of our Theorems~\ref{theorem:Main-1} and~\ref{theorem:Main-2}.

\begin{theorem}
\label{theorem:Main-3}
Let $\ell,d\geq 2$ be integers.
There is no $k$-regular-$\ell$-skew embedding $\RR^d\to\RR^N$ for
\[
N\leq (d-1)(k-\alpha(k))+(2^{\gamma(d)}-d-1)(\ell-\alpha(\ell))+(d+1)\ell+k-2,
\]
where $\alpha(c)$ denotes the number of ones in the dyadic expansion of $c$,
and $\gamma(d):=\lfloor\log_2d\rfloor+1$.
\end{theorem}

Using the notation of~\cite{stojanovic} the previous theorem can be stated as
\[
N_{k,\ell}(M)\geq N_{k,\ell}(\RR^d)
\geq (d-1)(k-\alpha(k))+(2^{\gamma(d)}-d-1)(\ell-\alpha(\ell))+(d+1)\ell+k-1,
\]
where $M$ denotes any smooth $d$-manifold.
It is an exercise to verify that the new bound presents a considerable 
improvement of the previously known bounds~\cite[Theorem 3.1]{stojanovic}.

\subsection{Definition and first bounds}
Here is the common generalization of ``$k$-regular maps'' and ``$\ell$-skew embeddings''.
\begin{definition}[Regular-skew embeddings]
  Let $k,\ell\geq 1$ be integers, and $M$ be a smooth $d$-dimensional manifold.  A smooth embedding
  $f\colon M\to\RR^N$ is {\bf $k$-regular-$\ell$-skew embedding} if for every
  $(x_1,\ldots,x_k,y_1,\ldots,y_{\ell})$ in $F(M,k+\ell)$ the affine subspaces
  \[
  \{f(x_1)\},\ldots,\{f(x_k)\},(\iota\circ df_{y_1})(T_{y_1}M),\ldots ,(\iota\circ
  df_{y_{\ell}})(T_{y_{\ell}}M)
  \]
  of $\RR^N$ are affinely independent, where $\iota$ has been defined in~\eqref{iota}.
\end{definition}

For $\ell=0$, the notion of the $k$-regular-$\ell$-skew embedding coincides with the notion of
affinely $(k-1)$-regular map.  On the other hand, for $k=0$ we get the notion of
$\ell$-skew embedding.

The following lower bound for $k$-regular-$\ell$-skew embeddings was obtained by
Stojanovi\'c~\cite[Theorem 3.1]{stojanovic} as a direct consequence of 
Theorem~\ref{the:BRS} due to Boltjanski{\u\i}, Ry{\v{s}}kov and {\v{S}}a{\v{s}}kin.

\begin{lemma}
  Let $k,\ell\geq 1$ be integers, and $M$ be a $d$-dimensional manifold, and $f\colon M\to\RR^N$
  be a $k$-regular-$\ell$-skew embedding.  Then $\lfloor\tfrac{k}{2}\rfloor
  d+\lfloor\tfrac{k-1}{2}\rfloor+(d+1)\ell\leq N$.
\end{lemma}

Combining the proof of the previous lemma and our Theorem~\ref{theorem:Main-1} we get the
following new and improved bound.

\begin{lemma}
  Let $k,\ell\geq 1$ be integers, $M$ be a $d$-dimensional manifold, and $f \colon M\to\RR^N$ be
  a $k$-regular-$\ell$-skew embedding.  Then $d(k-\alpha(k))+\alpha(k)+(d+1)-1\leq N$.
\end{lemma}
\begin{proof}
  Consider $(y_1,\ldots,y_{\ell})\in F(M,\ell)$. Let $U$ be the affine span of
  $(\iota\circ  df_{y_1})(T_{y_1}M)\cup\cdots\cup(\iota\circ df_{y_{\ell}})(T_{y_{\ell}}M)$.
  Denote, as before, by $a\colon\RR^N\to\RR^{N+1}$ the embedding sending
  $(t_1,\ldots,t_N)$ to $(t_1,\ldots,t_N,1)$.  The linear span of the affine subspace $a(U)$
  in $\RR^{N+1}$ will be denoted by $V$.

  Then the map \[g\colon
  M{\setminus}\{y_1,\ldots,y_{\ell}\}\to\RR^{N+1}/V\approx\RR^{N+1-(d+1)\ell}\] induced by
  the restriction of $a\circ f$ and the projection $\RR^{N+1}\to\RR^{N+1}/V$ is a
  $k$-regular map. Thus by Theorem~\ref{theorem:Main-1} we get that 
  $N+1-(d+1)\ell>d(k-\alpha(k))+\alpha(k)-1$, that is,
  \[
    d(k-\alpha(k))+\alpha(k)+(d+1)\ell-1\leq N.
  \]
\end{proof}

\subsection{A topological criterion}
As in the previous two sections, we derive a topological criterion for the
existence of a $k$-regular-$\ell$-skew embedding.  Again, $M$ denotes a smooth $d$-dimensional manifold
and $TM$ the tangent bundle over $M$.  Let us also assume that $k,\ell\geq 1$, since the
cases when $k=0$ or $\ell=0$ are already discussed.

In this section we consider two configuration spaces $F(M,k+\ell)$ and $F(M,\ell)$ with
actions of, respectively, $\Sym_k\times\Sym_{\ell}$ and $\Sym_{\ell}$ on them.  Let
\[
p\colon F(M,k+\ell)/(\Sym_k\times\Sym_{\ell})\to F(M,\ell)/\Sym_{\ell}
\] 
be the projection given by forgetting the first $k$ elements in $ F(M,k+\ell)$.

The group $\Sym_k\times\Sym_{\ell}$ can naturally be seen as a subgroup of the symmetric
group $\Sym_{k+\ell}$ that acts on $\RR^{k+\ell}$ by permuting coordinates.  The inclusion
$\Sym_k\times\Sym_{\ell}\subset\Sym_{k+\ell}$ induces an action of
$\Sym_k\times\Sym_{\ell}$ on $\RR^{k+\ell}$.  A new vector bundle we consider is $\gamma_{M,k,\ell}$ given by
\[
 \RR^{k+\ell}\longrightarrow F(M,k+\ell)\times_{(\Sym_k\times\Sym_{\ell})}\RR^{k+\ell}
\longrightarrow F(M,k+\ell)/{(\Sym_k\times\Sym_{\ell})}.
\]
If we denote by
\[
r \colon F(M,k+\ell)/(\Sym_k\times\Sym_{\ell})\to F(M,k+\ell)/\Sym_{k+\ell}
\]
the quotient map, then it is not hard to see that
\[
r^*\xi_{M,k+\ell}\cong\gamma_{M,k,\ell}.
\]

As before, the first step in obtaining our topological criterion is the existence of a
fiberwise linear monomorphism.

\begin{lemma}
  \label{lemma:TK-(k,l)-1}
  If there is a $k$-regular-$\ell$-skew embedding $M\to\RR^N$, then the
  $(k+(d+1)\ell)$-\allowbreak dimensional vector bundle
  $p^*T(F(M,\ell)/\Sym_{\ell})\oplus\gamma_{M,k,\ell}$ over
  $F(M,k+\ell)/(\Sym_k\times\Sym_{\ell})$ admits an $(N+1-k-(d+1)\ell)$-dimensional
  inverse.
\end{lemma}
\begin{proof}
  As in the proof of Lemma~\ref{lemma:TK-l-1} let $a\colon\RR^N\to\RR^{N+1}$ and
  $b\colon\RR^N\to\RR^{N+1}$ denote the embeddings
\[
 a(t_1,\ldots,t_N)=(t_1,\ldots,t_N,1)\quad \text{and} \quad  b(t_1,\ldots,t_N)=(t_1,\ldots,t_N,0).
\]

Let $f\colon M\to\RR^N$ be a $k$-regular-$\ell$-skew embedding.
Next we will construct a morphism of vector bundles
\[
\alpha \colon p^*T(F(M,\ell)/\Sym_{\ell})\oplus\gamma_{M,k,\ell}\longrightarrow F(M,k+\ell)/(\Sym_k\times\Sym_{\ell})\times\RR^{N+1}
\]
that is a monomorphism on the fibers.
Consider $\underline{z} \in F(M,k+\ell)$.  Let $\underline{x} \in F(M;k)$ and
$\underline{y} \in F(M,\ell)$ respectively be the elements obtained from $\underline{z}$
by ignoring the last $\ell$ coordinates and the first $k$ coordinates respectively.  Let
$q_{k+\ell} \colon F(M,k+\ell) \to F(M,k+\ell)/(\Sym_k\times\Sym_{\ell})$ and 
$q_{\ell} \colon F(M,\ell) \to F(M,\ell)/\Sym_{\ell}$ be the projections. An element $(v,w)$ in the fiber
of $p^*T(F(M,\ell)/\Sym_{\ell})\oplus\gamma_{M,k,\ell}$ over $q_{k,\ell}(\underline{z})$ is
given by elements $v \in T_{q_{\ell}(\underline{y})}(F(M,\ell)/\Sym_{\ell})$ and $w \in
(\gamma_{M,k,\ell})_{q_{k,\ell}(\underline{z})}$.  The vector $v$ is specified by a sequence of
vectors $v_i \in T_{y_i}M$ for $i = 1, \ldots, \ell$ and $w$ by an element
$\underline{\lambda} \in \RR^{k + \ell}$. Consider the element in $\RR^{N+1}$ given by
\begin{multline*}
\lambda_1(a\circ f(x_1)) + \cdots + \lambda_k(a\circ f(x_k))+
\lambda_{k+1}(a\circ f(y_1)) + \cdots + \lambda_{k+\ell}(a\circ f(y_{\ell}))
\\
+
b\circ df_{y_1}(v_1) + \cdots + b\circ df_{y_{\ell}}(v_{\ell}).
\end{multline*}
It defines an element in the fiber over $q_{k,\ell}(\underline{z})$
of the  trivial vector bundle bundle 
$F(M,k+\ell)/(\Sym_k\times\Sym_{\ell})\times\RR^{N+1} \to F(M,k+\ell)/(\Sym_k\times\Sym_{\ell})$  
which we declare to be the image of
$(v,w)$ under $\alpha$. One easily checks that $\alpha$ is a well-defined
morphism of vector bundles which is injective on each fiber.
Thus the vector bundle $p^*T(F(M,\ell)/\Sym_{\ell})\oplus\gamma_{M,k,\ell}$ admits the required inverse.
\end{proof}

A direct consequence of Lemma~\ref{lemma:TK-(k,l)-1} is the following criterion for the non-existence of a
$k$-regular-$\ell$-skew embedding expressed in terms of a dual Stiefel--Whitney class of the
vector bundle $p^*T(F(M,\ell)/\Sym_{\ell})\oplus\gamma_{M,k,\ell}$ over $F(\RR^d,k+\ell)/(\Sym_k\times\Sym_{\ell})$.

\begin{lemma}
  \label{lemma:TK-(k,l)-2}
  Let $M$ be a smooth $d$-dimensional manifold.  If the dual Stiefel--Whitney
  class \[\overline{w}_{r+1}(p^*T(F(M,\ell)/\Sym_{\ell})\oplus\gamma_{M,k,\ell})\] does not
  vanish, then there is no $k$-regular-$\ell$-skew embedding $M\to\RR^{r+k+\ell(d+1)-1}$.
\end{lemma}

\subsection{Proof of Theorem~\ref{theorem:Main-3}}

For the rest of this section assume that $M=\RR^d$.
Consider two embeddings $e_1, e_2 \colon \RR^d \to \RR^d$ whose images are disjoint. 
They induce in the obvious way a map
\[q\colon  F(\RR^d,k)/\Sym_k\times F(\RR^d,\ell)/\Sym_{\ell}\to F(\RR^d,k+\ell)/(\Sym_k\times\Sym_{\ell}).\]
One easily checks
\begin{equation}
\label{eq:001}
 q^*\gamma_{\RR^d,k,\ell}\cong\xi_{\RR^d,k}\times\xi_{\RR^d,\ell}.
\end{equation}

\begin{lemma}
\label{lemma:TK-(k,l)-3}
Let $\ell,k,d\geq 2$ be integers.  There exist isomorphisms of vector bundles over 
$F(\RR^d,k)/\Sym_k\times F(\RR^d,\ell)/\Sym_{\ell}$
\begin{eqnarray*}
 q^*\Big(p^*T(F(\RR^d,\ell)/\Sym_{\ell})\oplus\gamma_{\RR^d,k,\ell})\Big)
& \cong &
 \xi_{\RR^d,k}\times (T(F(\RR^d,\ell)/\Sym_{\ell})\oplus\xi_{\RR^d,\ell})
\\
& \cong &
\xi_{\RR^d,k}\times((d+1)\,\xi_{\RR^d,\ell}).
\end{eqnarray*}
\end{lemma}
\begin{proof}
  The first isomorphism follows from~\eqref{eq:001}, while the second isomorphism is a
  direct consequence of the isomorphism~\eqref{eq:TMofConfSpace}.
\end{proof}

\begin{theorem}
\label{theorem:Main-3.1}
Let $\ell,k,d\geq 2$ be integers.
The dual Stiefel--Whitney class
\[
\overline{w}_{(d-1)(k-\alpha(k))+(2^{\gamma(d)}-d-1)(\ell-\alpha(\ell))}(\xi_{\RR^d,k}\times((d+1)\,\xi_{\RR^d,\ell}))
\]
does not vanish.
\end{theorem}
\begin{proof}
In order to simplify the notation let us denote $r:=(d-1)(k-\alpha(k))$ and $s:=(2^{\gamma(d)}-d-1)(\ell-\alpha(\ell))$.

First, let us recall what has been proved so far:
\begin{itemize}
\item $\overline{w}_{r}(\xi_{\RR^d,k})\neq 0$, Theorem~\ref{theorem:Main-1.1}, and
\item $\overline{w}_{s}((d+1)\,\xi_{\RR^d,\ell})\neq 0$, Theorem~\ref{theorem:Main-2.1}.
\end{itemize}

Now we apply the product formula~\cite[Problem~4-A, page~54]{milnor-stasheff}:
\[
 \overline{w}\big(\xi_{\RR^d,k}\times((d+1)\,\xi_{\RR^d,\ell})\big)= \overline{w}(\xi_{\RR^d,k})\times\overline{w}((d+1)\,\xi_{\RR^d,\ell}),
\]
where $\times$ on the right hand side denotes the cross product in the cohomology.
Hence we get
\begin{multline*}
\overline{w}_{r+s}\big(\xi_{\RR^d,k}\times((d+1)\,\xi_{\RR^d,\ell})\big)
\\
 =\overline{w}_{r}(\xi_{\RR^d,k})\times \overline{w}_{s}((d+1)\,\xi_{\RR^d,\ell})
+
\sum_{i+j=r+s,i\neq r,j\neq s}
\overline{w}_{i}(\xi_{\RR^d,k})\times \overline{w}_{j}((d+1)\,\xi_{\RR^d,\ell}).
\end{multline*}
Since $\overline{w}_{r}(\xi_{\RR^d,k})\times \overline{w}_{s}((d+1)\,\xi_{\RR^d,\ell})\neq 0$ it follows by the K\"unneth formula that
\[
 \overline{w}_{r+s}\big(\xi_{\RR^d,k}\times((d+1)\,\xi_{\RR^d,\ell})\big)\neq 0.\vspace{-6mm}
\]
\end{proof}

Now Theorem~\ref{theorem:Main-3}  follows from Lemma~\ref{lemma:TK-(k,l)-2}, Lemma~\ref{lemma:TK-(k,l)-3} 
and Theorem~\ref{theorem:Main-3.1}.


\providecommand{\noopsort}[1]{}


\end{document}